\documentclass[12pt]{amsart}
\usepackage{latexsym, amsthm, amscd, euscript, float, subfig}
\usepackage{color}

\usepackage{amssymb}
\usepackage{amsmath}
\usepackage{graphicx}

{}

\newtheorem{theorem}{Theorem}[section]
\newtheorem{lemma}[theorem]{Lemma}
\newtheorem{corollary}[theorem]{Corollary}

\newtheorem{remark}[theorem]{Remark}
\newtheorem{example}[theorem]{Example}

\numberwithin{equation}{section}
\newcommand{\IR}{\mathbb{R}}
\newcommand{\IB}{\mathbb{B}}

\newcommand{\Bn}{ {\mathbb{B}^n} }
\newcommand{\Rn}{ {\mathbb{R}^n} }

\newcommand{\beq}{\begin{equation}}
\newcommand{\eeq}{\end{equation}}

\renewcommand{\tanh}{\,\textnormal{tanh}}
\renewcommand{\sinh}{\,\textnormal{sinh}}

\begin{document}
\vspace*{-2cm}
\title[Geometry of the Cassinian metric and its inner metric]
{Geometry of the Cassinian metric and its inner metric}

\date{\today}

\author[Z. Ibragimov]{Zair Ibragimov
}
\address{Zair Ibragimov, Department of Mathematics, California State University,
Fullerton, CA, 92834
}
\email{zibragimov@fullerton.edu}

\author[M. R. Mohapatra]{Manas Ranjan Mohapatra}
\address{Manas Ranjan Mohapatra, Discipline of Mathematics,
Indian Institute of Technology Indore,
Indore 452 017, India
}
\email{mrm.iiti@gmail.com}

\author[S. K. Sahoo]{Swadesh Kumar Sahoo}
\address{Swadesh Kumar Sahoo, Discipline of Mathematics,
Indian Institute of Technology Indore,
Indore 452 017, India}
\email{swadesh@iiti.ac.in}

\author[X.-H. Zhang]{Xiaohui Zhang}
%${}^*$}
\address{Xiaohui Zhang, Department of Physics and  Mathematics, University of Eastern Finland, 80101 Joensuu, Finland}
%\curraddr{}
\email{xiaohui.zhang@uef.fi}
%\thanks{${}^*$ The author's research was supported by the Academy of Finland project 268009.}

\begin{abstract}
The Cassinian metric and its inner metric have been studied for subdomains
of the $n$-dimensional Euclidean space $\IR^n$ ($n\ge 2$) by the first named author. In this paper we obtain
various inequalities between the Cassinian metric and other related metrics in some
specific subdomains of $\IR^n$. Also, a sharp distortion property of the Cassinian metric
under M\"obius transformations of the unit ball is obtained.
\\


\end{abstract}

\maketitle
\thispagestyle{empty}
\section{Introduction}
One of the aspects of hyperbolic geometry deals with the comparison of
the hyperbolic metric with the so-called hyperbolic-type
metrics. Secondly, invariance and distortion properties of hyperbolic-type metrics
under conformal maps (M\"obius transformations in higher dimensions) also play significant roles
in geometric function theory. In recent years, many authors have contributed to
the study of hyperbolic-type metrics. Some of the familiar hyperbolic-type metrics are
the quasihyperbolic metric \cite{GO79,GP76}, the distance ratio metric \cite{Vuo85},
the Apollonian metric \cite{Bea98,BI08,Has03,HI05,HI07,Ibr03,Ibr2003}, the Seittenranta metric \cite{Sei99},
the Ferrand metric \cite{Fer88,HIL06,HIMPS07}, the K--P metric \cite{HIMPS07,HIM08,KP94}, the Cassinian metric \cite{Ibr09},
the visual angle metric \cite{KLVW}, and the triangular ratio metric \cite{CHKV}.
These metrics are also referred to as the {\em relative metrics} since that they are defined in a proper subdomain of the Euclidean space $\IR^n$, $n\ge 2$, relative to its boundary.
A more general form of relative metrics has been considered by P. H\"ast\"o in \cite[Lemma~6.1]{Has02}.
In this paper we study geometric properties of the Cassinian metric
by comparing it with the hyperbolic, distance ratio, and visual angle metrics.
For a quick overview on these metrics, the reader can refer to the next section.
We also discuss the quasi-invariance (distortion)
property of the Cassinian metric under M\"obius transformations of the unit
ball. Finally, we compute the inner metric of the Cassinian metric, the so-called
{\em inner Cassinian metric}, in some specific subdomains of $\mathbb{R}^n$ and study some of its basic properties.

\section{Preliminaries}\label{Intro}

Throughout the paper $D$ denotes an arbitrary, proper subdomain of the Euclidean space $\IR^n$, i.e., $D\subsetneq \IR^n$. 
The Euclidean distance between $x,y\in\IR^n$ is denoted by $|x-y|$. The standard Euclidean norm of a point 
$x\in\IR^n$ is denoted by $|x|$. Given $x\in\IR^n$ and $r>0$, the open ball centered at $x$ and of radius $r$ is denoted by $B(x,r)=\{y\in\IR^n\colon\, |x-y|<r\}$. 
The unit ball in $\IR^n$ is denoted by $\IB^n$. 
The closed line segment between two points $x$ and $y$ in $\mathbb{R}^n$ is denoted by $[x,y]$.
Given $x\in D$, the distance $\delta_D(x)$ from $x$ to the boundary $\partial D$ of $D$ is given by
$$
\delta_D(x)=\inf\big\{|x-\xi|\colon\, \xi\in\partial D\big\}.
$$
For real numbers $r$ and $s$, we set $r\vee s=\max\{r,\, s\}$ and $r\wedge s=\min\{r,\, s\}$.

The {\it Cassinian metric} $c_D$ of the domain $D$ is defined as
$$
c_D (x,y)=\sup_{p\in \partial D} \frac{|x-y|}{|x-p||p-y|} .
$$
This metric was first introduced and studied in \cite{Ibr09}.
However, a more general form of this metric was considered by P. H\"ast\"o (see \cite[Lemma~6.1]{Has02}).
%Clearly, the supremum is attained at some point $p\in \partial D$.

The {\it distance ratio metric $j_D$} is defined by
$$
j_D(x,y)=\log\left(1+\frac{|x-y|}{\delta_D(x)\wedge\delta_D(y)}\right).
$$
The above form of the metric $j_D$, which was first considered in \cite{Vuo85},
is a slight modification of the original distance ratio metric introduced in \cite{GO79,GP76}.
This metric has been widely studied in the literature; see, for instance, \cite{Vuo88}.

The {\it{hyperbolic metric}} $\rho_{\IB^n}$ of the unit ball $\IB^n$ is given by
$$
\rho_{\IB^n}(x,y)=\inf_\gamma\int_\gamma \frac{2|{\rm d}z|}{1-|z|^2},
$$
where the infimum is taken over all rectifiable curves $\gamma\subset \IB^n$
joining $x$ and $y$.

The {\it{visual angle metric}} $v_D$, introduced in \cite{KLVW}, is defined by
$$
v_D (x,y)=\sup \{\angle(x,z,y):z\in \partial D\}.
$$
We also consider the quantity $p_D$,
$$
p_D(x,y)=\frac{|x-y|}{\sqrt{|x-y|^2+4\delta_D(x)\delta_D(y)}}.
$$
Note that the quantity $p_D$, which was first considered in \cite{CHKV}, does not
define a metric (see \cite[Remark~3.1]{CHKV}).
However, it has a nice connection with the hyperbolic metric, $\rho_{\mathbb{H}^2}$, of
the upper half-plane $\mathbb{H}^2=\{(x,y)\in\mathbb R^2\colon\, y>0\}$. Namely,
$$
p_{\mathbb{H}^2}(z_1,z_2)=\tanh\frac{\rho_{\mathbb{H}^2}(z_1,z_2)}{2}=\frac{|z_1-z_2|}{|z_1-\bar{z}_2|}, \quad z_1,z_2\in \mathbb{H}^2,
$$
where $\bar{z}_2$ is the reflection of $z_2$ with respect to the real line $\mathbb{R}$ (see \cite{CHKV}).
Hence it is natural to ask whether the quantity $p_D$ is comparable with hyperbolic-type metrics such as, the Cassinian metric $c_D$,
in more general domains $D$.

\section{Comparison of the Cassinian metric with other related quantities}

This section is devoted to finding upper and lower bounds for the Cassinian metric in terms of
the quantities, defined in Section~\ref{Intro}, in some specific domains.
We begin with the comparison of the Cassinian and hyperbolic metrics of the unit ball
$\IB^n$. Recall that for all $x,y\in \Bn$
\begin{equation}\label{fsinh}
\sinh\left(\frac{\rho_{\Bn}(x,y)}{2}\right)=\frac{|x-y|}{\sqrt{(1-|x|^2)(1-|y|^2)}},
\end{equation}
(see, for example, \cite[p.~40]{Bea95}).

\begin{theorem}
For $x,y\in\Bn$, we have
\begin{equation}\label{rho+c}
\sinh\left(\frac{\rho_{\Bn}(x,y)}{2}\right)\leq c_{\Bn}(x,y).
\end{equation}
\end{theorem}

\begin{proof}
Without loss of generality, we may assume that $|y|\geq|x|$.  It is trivial that the inequality \eqref{rho+c} holds for $y=0$,
since $x=0$ in this case also.

Hence we assume that $y\neq0$. It is easy to see that
\allowdisplaybreaks\begin{align*}
\inf_{z\in\partial \Bn}|x-z||y-z| & \leq\left|x-\frac{y}{|y|}\right|\left|y-\frac{y}{|y|}\right|\\
                                  & =(1-|y|)\left|x-\frac{y}{|y|}\right|\\
                                  & \leq (1-|y|)(1+|x|)\\
                                  & \leq \sqrt{(1-|x|^2)(1-|y|^2)},
\end{align*}
where the last inequality follows since $|x|\le |y|$. Now, the formula (\ref{fsinh}) easily yields
$$
\sinh\left(\frac{\rho_{\Bn}(x,y)}{2}\right) %& = \frac{|x-y|}{\sqrt{(1-|x|^2)(1-|y|^2)}}\\
\leq \frac{|x-y|}{\inf\limits_{z\in\partial \Bn}|x-z||y-z|} =  c_{\Bn}(x,y).
$$
Hence the proof is complete.
\end{proof}

\begin{remark}
Inequality \eqref{rho+c} is sharp in the following sense. For $0$ and
$x$ in $\Bn$, we use the formulae 
\begin{equation}\label{sec3-eq1}
c_{\IB^n}(0,x)=\frac{|0-x|}{|1-0||1-x|}=\frac {|x|}{1-|x|}
\end{equation}
$($see \cite[Example 3.9(B)]{Ibr09}$)$ and $(\ref{fsinh})$. It follows that
$$
\frac{\sinh\left(\displaystyle\frac{\rho_{\Bn}(0,x)}{2}\right)}{c_{\Bn}(0,x)}=\frac{1-|x|}{\sqrt{1-|x|^2}}
$$
approaches $1$ as $x$ approaches $0$.
\end{remark}
It is well known that $\sinh x\geq x$ for all $x\geq0$.
This leads to
\begin{corollary}\label{cor}
For $x,y\in\Bn$, we have the following sharp inequality
$$
\rho_{\Bn}(x,y)\leq 2c_{\Bn}(x,y).
$$
\end{corollary}

Next, we compare the Cassinian metric and the distance ratio metric in proper subdomains of $\Rn$.
\begin{theorem}\label{sec3-thm3}
Let $D$ be a proper subdomain of $\Rn$ and let $x,y\in D$. Then
$$j_D(x,y)\leq \big(|x-y|+\delta_D(x)\wedge \delta_D(y)\big)c_D(x,y).
$$
\end{theorem}

\begin{proof}
We may assume that $\delta_D(x)\wedge \delta_D(y)=\delta_D(x)$.
Choose $z\in\partial D$ such that $\delta_D(x)=|x-z|$. By the triangle inequality, we have that
$$
\inf_{p\in\partial D}|x-p||y-p| \leq |x-z||y-z| \leq \delta_D(x)(|x-y|+\delta_D(x)),
$$
and
\allowdisplaybreaks\begin{align*}
c_D(x,y) & \geq \frac{|x-y|}{\delta_D(x)(|x-y|+\delta_D(x))}\\
           & \geq \frac{1}{|x-y|+\delta_D(x)}\log\left(1+\frac{|x-y|}{\delta_D(x)}\right)\\
           & = \frac{1}{|x-y|+\delta_D(x)}j_D(x,y).
\end{align*}
This completes the proof of our theorem.
\end{proof}

\begin{corollary}
For $x,y\in \Bn$, we have
$$
j_\Bn(x,y)\leq (1+|x|\wedge|y|)c_\Bn(x,y)\leq 2c_\Bn(x,y).
$$
In particular,
$$
j_\Bn(0,x)\leq c_\Bn(0,x).
$$
\end{corollary}
\begin{proof}
Since $x,y\in\Bn$, by Theorem~\ref{sec3-thm3} we observe that
\allowdisplaybreaks\begin{align*}
|x-y|+\delta_D(x)\wedge \delta_D(y) & \leq |x|+|y|+(1-|x|)\wedge (1-|y|)\\
           & = |x|+|y|+1-|x|\vee|y|\\
           & = 1+|x|\wedge|y|.
\end{align*}
The desired inequalities follow.

\end{proof}
\begin{lemma}
For all $x,y\in \IB^n$ with $|x|\vee |y|\leq\lambda<1$
we have
\begin{equation}\label{cj}
c_{\IB^n}(x,y)\le \frac{1}{(1-\lambda)^2}j_{\IB^n}(x,y).
\end{equation}
\end{lemma}

\begin{proof}
Without loss of generality, we assume that $|y|=|x|\vee |y|\leq\lambda$. For any $w\in \partial \IB^n$, we have
$$ |x-w||w-y|\ge (1-\lambda)^2,
$$
and hence,
\begin{equation}\label{eqn2}
(1-\lambda)^2c_{\IB^n}(x,y)\le |x-y|.
\end{equation}
Now,
\begin{eqnarray*}
j_{\IB^n}(x,y)
&=& \log \left(1+\frac{|x-y|}
{\delta_{\IB^n}(x)\wedge\delta_{\IB^n}(y)}\right)=\log \left(1+\frac{|x-y|}{1-|y|}\right)\\
&\ge & \frac{\displaystyle\frac{2|x-y|}{1-|y|}}{2+\displaystyle\frac{|x-y|}{1-|y|}}
\quad \left(\because \log(1+t)\ge \frac{2t}{2+t} \mbox{ for }t>0\right)\\
&= & \frac{2|x-y|}{2-2|y|+|x-y|}
\ge  |x-y|\ge (1-\lambda)^2 c_{\IB^n}(x,y),
\end{eqnarray*}
where the last two inequalities follow from the inequalities $|x-y|\leq 2|y|$ and (\ref{eqn2}) respectively.
\end{proof}

The next lemma describes the relations between the Cassinian metric and the visual angle metric of the unit ball.

\begin{lemma}
The following inequalities hold.
\begin{enumerate}
\item For $x,y\in \Bn$ we have
$$
\frac{v_\Bn(x,y)}{2} \leq \tan\frac{v_\Bn(x,y)}{2} \leq c_\Bn(x,y).
$$
\item For all $x,y\in \IB^2$ with $|x|\vee |y|\leq\lambda<1$
we have
$$c_{\IB^2}(x,y)\le \frac{2(3+\lambda^2)}{3(1-\lambda^2)(1-\lambda)^2} v_{\IB^2} (x,y).
$$
\end{enumerate}
\end{lemma}
\begin{proof}
Combining the inequality \cite[Theorem~3.11]{KLVW}
$$
\tan\frac{v_\Bn(x,y)}{2}\leq \sinh\frac{\rho_\Bn(x,y)}{2}
$$
and the inequality \eqref{rho+c}, we have that
$$
\tan\frac{v_\Bn(x,y)}{2} \leq c_\Bn(x,y).
$$
It is clear that
$$
\frac{v_\Bn(x,y)}{2} \leq \tan\frac{v_\Bn(x,y)}{2}.
$$
This proves the first part.

For the proof of the second part,
we combine the inequality \cite[Theorem~3.9]{CHKV}
$$
j_{\IB^2} (x,y)\le \frac {2(3+\lambda^2)}{3(1-\lambda^2)}v_{\IB^2}(x,y).
$$
with (\ref{cj}), we obtain
$$c_{\IB^2}(x,y)\le \frac{2(3+\lambda^2)}{3(1-\lambda^2)(1-\lambda)^2} v_{\IB^2} (x,y).
$$
Thus, the proof of our lemma is complete.
\end{proof}

Next, we compare the Cassinian metric $c_D$ with the quantity $p_D$.

\begin{theorem}\label{thrm3}
Let $x,y\in D\subsetneq \mathbb{R}^n$. Then
$$
p_{D}(x,y)\le \sqrt{2}\Big(\delta_D(x)\wedge\delta_D(y)\Big) c_D(x,y).
$$
\end{theorem}
\begin{proof}
Fix $x,y\in D$ and let $s=\delta_D(x)\wedge\delta_D(y)$. Then
\begin{eqnarray*}
p_D(x,y) &=& \frac{|x-y|}{\sqrt{|x-y|^2+4\delta_D(x)\delta_D(y)}}\le \frac{|x-y|}{\sqrt{|x-y|^2+(2s)^2}}\\
&\le&  \frac{\sqrt{2}|x-y|}{|x-y|+2s}\le \frac{\sqrt{2}|x-y|}{|x-y|+s}\\
&=& \sqrt{2}\Big(\delta_D(x)\wedge\delta_D(y)\Big) \frac{|x-y|}{\Big(\delta_D(x)\wedge\delta_D(y)\Big)\Big(|x-y|+\big(\delta_D(x)\wedge\delta_D(y)\big)\Big)}\\
&=& \sqrt{2}\Big(\delta_D(x)\wedge\delta_D(y)\Big)\Bigg[\frac{|x-y|}{\delta_D(x)\big(|x-y|+\delta_D(x)\big)}\vee\frac{|x-y|}{\delta_D(y)\big(|x-y|+\delta_D(y)\big)}\Bigg]\\
&\le& \sqrt{2}\Big(\delta_D(x)\wedge\delta_D(y)\Big) c_D(x,y),
\end{eqnarray*}
where the second inequality follows from \cite[1.58~(13)]{AVV} and the last inequality follows from \cite[Lemma 3.4]{Ibr09}.
\end{proof}

\begin{remark}
Observe that if we take the domain $D$ in Theorem~\ref{thrm3} to be the unit ball $\mathbb{B}^n$, then we can see that $p_D(x,y)\le \sqrt{2}c_D(x,y)$.
In fact, if $D$ is a bounded domain in $\mathbb{R}^n$, then $p_D(x,y)\le \big(\operatorname{diam}(D)/\sqrt{2}\big)c_D(x,y)$.
\end{remark}

\section{Distortion of the Cassinian metric under M\"obius transformations of the unit ball}

In this section we study distortion properties of the Cassinian metric $c_{\IB^n}$ of the unit ball $\IB^n$ under M\"obius transformations of $\IB^n$.
Note that the M\"obius transformations of $\IB^n$ preserve the hyperbolic metric $\rho_{\IB^n}$.

Let $\phi$ be a M\"obius transformation with $\phi(\IB^n)=\IB^n$ and put $a=\phi(0)$. If $a=0$, then $\phi$ is an orthogonal matrix, i.e., $|\phi(x)|=|x|$ for each $x\in\IB^n$. In particular,  $\phi$ preserves the Cassinian metric. That is,
\begin{equation}\label{mob}
c_{\IB^n}\big(\phi(x),\phi(y)\big)=c_{\IB^n}(x,y)\qquad\text{for all}\quad x,y\in\IB^n.
\end{equation}

Suppose now that $a\neq 0$. Let $\sigma$ be the inversion in the sphere $\mathbb S^{n-1}(a^\star,r)$, where
$$
a^\star=\frac {a}{|a|^2}\qquad\text{and}\qquad r=\sqrt{|a^\star|^2-1}=\frac {\sqrt{1-|a|^2}}{|a|}.
$$
Note that the sphere $\mathbb S^{n-1}(a^\star,r)$ is orthogonal to $\partial\IB^n$ and that $\sigma(a)=0$.
In particular, $\sigma$ is a M\"obius transformation with $\sigma(\IB^n)=\IB^n$ and
$\sigma(a)=0$. Recall that
$$
\sigma(x)=a^\star+\Big(\frac {r}{|x-a^\star|}\Big)^2\big(x-a^\star\big).
$$
Then $\sigma\circ\phi$ is an orthogonal matrix (see, for example, \cite[Theorem 3.5.1(i)]{Bea95}). %Since $\sigma^2=\operatorname{id}$, we have $\phi=\sigma\circ A$.
In particular,
\begin{equation}\label{mob1}
\Big|\sigma\big(\phi(x)\big)-\sigma\big(\phi(y)\big)\Big|=|x-y|.
\end{equation}
We will need the following property of $\sigma$ (see, for example, \cite[p. 26]{Bea95}):
\begin{equation}\label{mob2}
|\sigma(x)-\sigma(y)|=\frac {r^2|x-y|}{|x-a^\star||y-a^\star|}.
\end{equation}
It follows from (\ref{mob1}) and (\ref{mob2}) that
$$
|x-y|=\Big|\sigma\big(\phi(x)\big)-\sigma\big(\phi(y)\big)\Big|=\frac {r^2|\phi(x)-\phi(y)|}{|\phi(x)-a^\star||\phi(y)-a^\star|}=\frac {(|a^\star|^2-1)|\phi(x)-\phi(y)|}{|\phi(x)-a^\star||\phi(y)-a^\star|},
$$
or equivalently,
$$
|\phi(x)-\phi(y)|=\frac {|\phi(x)-a^\star||\phi(y)-a^\star|}{|a^\star|^2-1}|x-y|.
$$
In particular, for all $x,y\in\IB^n$ and $\eta\in\partial\IB^n$ we have
\begin{equation}\label{mob3}
\frac {|\phi(x)-\phi(y)|}{|\phi(x)-\phi(\eta)||\phi(y)-\phi(\eta)|}=\frac {|x-y|}{|x-\eta||y-\eta|}\cdot\frac {|a^\star|^2-1}{|\phi(\eta)-a^\star|^2}.
\end{equation}
Note that since $\phi(\eta)\in\partial\IB^n$ and $|a^\star|>1$, we have
$$
|a^\star|-1\leq|\phi(\eta)-a^\star|\leq |a^\star|+1
$$
and hence
\begin{equation}\label{mob4}
\frac {1-|a|}{1+|a|}=\frac {|a^\star|-1}{|a^\star|+1}\leq \frac {|a^\star|^2-1}{|\phi(\eta)-a^\star|^2}\leq \frac {|a^\star|+1}{|a^\star|-1}=\frac {1+|a|}{1-|a|}.
\end{equation}

Now given $x,y\in\IB^n$, there exist $\eta_1\in\partial\IB^n$ and $\eta_2\in\partial\IB^n$ such that
$$
c_{\IB^n}\big(\phi(x),\phi(y)\big)=\frac {|\phi(x)-\phi(y)|}{|\phi(x)-\phi(\eta_1)||\phi(y)-\phi(\eta_1)|}\quad\text{and}\quad c_{\IB^n}(x,y)=\frac {|x-y|}{|x-\eta_2||y-\eta_2|}.
$$
Using (\ref{mob3}) and (\ref{mob4}) we obtain
$$
c_{\IB^n}\big(\phi(x),\phi(y)\big)=\frac {|x-y|}{|x-\eta_1||y-\eta_1|}\cdot
\frac {|a^\star|^2-1}{|\phi(\eta_1)-a^\star|^2}\leq \frac {1+|a|}{1-|a|}c_{\IB^n}(x,y)
$$
and
$$
c_{\IB^n}(x,y)=\frac {|\phi(x)-\phi(y)|}{|\phi(x)-\phi(\eta_2)||\phi(y)-\phi(\eta_2)|}\cdot\frac {|\phi(\eta_2)-a^\star|^2}{|a^\star|^2-1}\leq \frac {1+|a|}{1-|a|}c_{\IB^n}\big(\phi(x),\phi(y)\big).
$$

The constant $(1+|a|)/(1-|a|)$ can be attained for all M\"obius transformations $\phi$ with $\phi(\IB^n)=\IB^n$ and $a=\phi(0)$. 
To see this, it suffices to consider the map $\sigma$ with $\sigma(a)=0$ for $a\in[0,e_1]\setminus\{0,e_1\}$ with 
$e_1=(1,0,\cdots,0)\in \IR^n$ since $c_{\IB^n}$ is invariant under orthogonal transformations. 
Choose $x=0$ and $y=te_1, -1<t<0$. Then we have that
$$
\sigma(x)=a\qquad \mbox{and}\qquad \sigma(y)=\frac{|a|-t}{1-|a|t}e_1\in[a,e_1]\setminus\{a,e_1\}.
$$
It is easy to see by the formula (\ref{sec3-eq1}) that
$$
c_{\IB^n}(x,y)=c_{\IB^n}(0,te_1)=-\frac{t}{1+t}.
$$
Furthermore, it follows from \cite[Example~3.9(B)]{Ibr09} that
$$c_{\IB^n}(re_1,se_1)=\frac{s-r}{(1-r)(1-s)},\quad 0\le r<s<1.
$$
This gives
\begin{align*}
c_{\IB^n}(\sigma(x),\sigma(y))=c_{\IB^n}\left(a,\frac{|a|-t}{1-|a|t}e_1\right)
=\frac{\displaystyle\frac{|a|-t}{1-|a|t}-|a|}{(1-|a|)\left(1-\displaystyle\frac{|a|-t}{1-|a|t}\right)}
=-\frac{t}{1+t}\frac{1+|a|}{1-|a|}.
\end{align*}
Therefore, we get
$$
\frac{c_{\IB^n}(\sigma(x),\sigma(y))}{c_{\IB^n}(x,y)}=\frac{1+|a|}{1-|a|}.
$$

Thus, we have proved the following theorem.
\begin{theorem}
Let $\phi$ be a M\"obius transformation with $\phi(\IB^n)=\IB^n$. Then
$$
\frac {1-|\phi(0)|}{1+|\phi(0)|}c_{\IB^n}(x,y)\leq c_{\IB^n}\big(\phi(x),\phi(y)\big)\leq\frac {1+|\phi(0)|}{1-|\phi(0)|}c_{\IB^n}(x,y)
$$
for all $x,y\in\IB^n$. The equalities in both sides can be attained.
\end{theorem}

\section{The inner Cassinian metric}
Let $D\subsetneq \IR^n$ and $\gamma$ be a rectifiable curve in $D$.
We define the Cassinian length of $\gamma$ as
$$
c_D(\gamma)=\sup \sum_{i=0}^{n-1} c_D(\gamma(t_i),\gamma(t_{i+1}))
$$
where the supremum is taken over all partitions $(t_i)_{i=1}^n$ of $I=[a,b]$
with $t_1=a$ and $t_n=b$. Then the inner Cassinian metric is defined as
$$
\tilde{c}_D(x,y)=\inf_\gamma c_D(\gamma)=\inf_\gamma \int_\gamma \frac{|\rm{d} z|}{(\delta_D(z))^2},
$$
where the infimum is taken over all rectifiable curves $\gamma\subset D$ connecting $x$ and $y$ (see \cite{Ibr09}).
First, we establish the monotonicity property of the inner Cassinian metric.

\begin{lemma}\label{mono}
The inner Cassinian metric is monotonic with respect to domains. That is, if $D\subset D'$, then $\tilde{c}_{D'}(x,y)\leq\tilde{c}_D(x,y)$ for all $x,y\in D$.
\end{lemma}
\begin{proof}
Given $x,y\in D$, we have
$$
\tilde{c}_D(x,y)=\inf_\gamma c_D(\gamma),
$$
where the infimum is taken over all rectifiable curves $\gamma\subset D$ connecting $x$ and $y$. Since the Cassinian metric is monotonic (\cite[Corollary 3.2]{Ibr09}), $c_D(\gamma)\geq c_{D'}(\gamma)$ for all such $\gamma$ and, consequently,
$$
\inf_\gamma c_D(\gamma)\geq\inf_\gamma c_{D'}(\gamma).
$$
Since each such $\gamma$ also connects $x$ and $y$ in $D'$, we have
$$
\tilde{c}_{D'}(x,y)=\inf_\gamma c_{D'}(\gamma)\leq\inf_\gamma c_D(\gamma)=\tilde{c}_D(x,y),
$$
completing the proof.
\end{proof}

Next, we compute the inner Cassinian metrics in some special cases.
\begin{example}
For the punctured space $D=\IR^n\setminus \{0\}$, the inner Cassinian metric $\tilde{c}_D$ is
same as the Cassinian metric $c_D$ and is given by the formula
$$
\tilde{c}_D(x,y)=c_D(x,y)=\frac{|x-y|}{|x||y|}.
$$
\end{example}
To see this, let $f(\xi)=\xi/|\xi|^2$ be the inversion about the unit sphere $\mathbb S^{n-1}(0,1)=\{\xi\in\IR^n\colon\, |\xi|=1\}$. Then $f(D)=D$ and that $f$ is an isometry between $(D,c_D)$ and $(D,|-|)$,
where $|-|$ is the Euclidean distance in $D$ (see, \cite[Example 3.9(A)]{Ibr09}). Since the inner metric of the Euclidean metric in $D$ is the same as the Euclidean metric itself and since $(D,c_D)$ is isometric to $(D,|-|)$, we conclude that $(D,\tilde{c}_D)$ is isometric to $(D,|-|)$.
Hence $\tilde{c}_D$ is same as the Cassinian metric $c_D$. In particular, it follows from \cite[(3.1.5)]{Bea95} that
$$
\tilde{c}_D(x,y)=c_D(x,y)=|f(x)-f(y)|=\frac{|x-y|}{|x||y|}
$$
for all $x,y\in D$, as required.

\begin{example}\label{exam2}
For each $x\in \IB^n$, we have
$$
\tilde{c}_{\IB^n}(0,x)=c_{\IB^n}(0,x)=\frac{|x|}{1-|x|}.
$$
\end{example}
It follows from \cite[Theorem 3.8]{Ibr09} that the line segment $[0,x]$ is a Cassinian geodesic so that its
Cassinian length is equal to $c_{\IB^n}(0,x)$. That is,
$$
c_{\IB^n}([0,x])=c_{\IB^n}(0,x)=\frac {|x|}{1-|x|},
$$
where the last equality is derived in (\ref{sec3-eq1}).
Therefore,
$$
c_{\IB^n}(0,x)\leq \tilde{c}_{\IB^n} (0,x)=\inf_\gamma c_{\IB^n}(\gamma)\leq c_{\IB^n}([0,x]).
$$
Hence $\tilde{c}_{\IB^n}(0,x)=c_{\IB^n}(0,x)$, as required.

The following corollary is an easy consequence of Lemma~\ref{mono} and Example~\ref{exam2}.

\begin{corollary}
Given $x\in D\setminus \{\infty\}$, we have
$$
\tilde{c}_D(x,y)\leq\frac{|x-y|}{\delta_D(x)(\delta_D(x)-|x-y|)}
$$
for all $y\in D$ with $|x-y|<\delta_D(x)$.
\end{corollary}
\begin{proof}
Set $B=B(x,\delta_D(x))$. Then as in Example~\ref{exam2} we obtain
$$
\tilde{c}_B(x,y)=\frac{|x-y|}{\delta_D(x)(\delta_D(x)-|x-y|)}
$$
for any $y\in B$. Now the conclusion follows from Lemma~\ref{mono}.
\end{proof}

\subsection*{Acknowledgment}
The last author is supported by the Academy of Finland project 268009.


\begin{thebibliography}{99}
%% Style of  articles
\bibitem{AVV}  {G. D. Anderson, M. K. Vamanamurthy, and M. K.
Vuorinen,}
{\em Conformal Invariants, Inequalities, and Quasiconformal Maps},
John Wiley \& Sons, Inc., 1997.

\bibitem{Bea95} {A. F. Beardon},
{\em Geometry of discrete groups},
Springer-Verlag, New York, 1995.

\bibitem{Bea98}  {A. F. Beardon},
{\em The Apollonian metric of a domain in $\mathbb{R}^n$}. In:
P. Duren, J. Heinonen, B. Osgood, and B. Palka (Eds.)
\textit{Quasiconformal mappings and analysis} (Ann Arbor, MI, 1995),
pp. 91--108.
Springer-Verlag, New York, 1998.

\bibitem{BI08}{M. Borovikova and Z. Ibragimov,}
{\em Convex bodies of constant width and the Apollonian metric},
Bull. Malays. Math. Sci. Soc., {\bf 31} (2) (2008), 117--128.

\bibitem{CHKV} { J. Chen, P. Hariri, R. Kl\'en} and {M. Vuorinen},
{\em Lipschitz conditions, Triangular ratio metric and Quasiconformal maps},
 Ann. Acad. Sci. Fenn. Math., 2015 (to appear),
doi:10.5186/aasfm.2015.4039, arXiv:1403.6582 [math.CA].

\bibitem{Fer88} {J. Ferrand},
{\em A characterization of quasiconformal mappings by the behavior of a
function of three points}. In: I. Laine, S. Rickman, and T. Sorvali (Eds.)
\emph{Proceedings of the 13th Rolf Nevalinna Colloquium} (Joensuu, 1987),
pp. 110--123. Lecture Notes in Mathematics Vol. 1351, Springer-Verlag, New York, 1988.

\bibitem{GO79} {F. W. Gehring and B. G. Osgood,}
{\em Uniform domains and the quasihyperbolic metric},
{J. Analyse Math.,} {\bf 36} (1979), 50--74.

\bibitem{GP76} {F. W. Gehring and B. P. Palka,}
{\em Quasiconformally homogeneous domains},
{J. Analyse Math.,} {\bf 30} (1976), 172--199.

\bibitem{Has02} {P. H\"ast\"o,}
{\em A new weighted metric: the relative metric. I},
{J. Math. Anal. Appl.}, {\bf 274}(1), (2002), 38--58.

\bibitem{Has03} {P. H\"ast\"o,}
{\em The Apollonian metric: uniformity and quasiconvexity,}
{Ann. Acad. Sci. Fenn. Math.}, {\bf 28} (2) (2003), 385--414.

\bibitem{HI05}{P. H\"ast\"o and Z. Ibragimov,}
{\em Apollonian isometries of planar domains are M\"obius mappings},
J. Geom. Anal., {\bf 15} (2) (2005), 229--237.

\bibitem{HI07}{P. H\"ast\"o and Z. Ibragimov,}
{\em Apollonian isometries of regular domains are M\"obius mappings},
Ann. Acad. Sci. Fenn., Ser. A I Math., {\bf 32} (1) (2007), 83--98.

\bibitem{HIL06}{P. H\"ast\"o, Z. Ibragimov and H. Linden,}
{\em Isometries of relative metrics},
Comput. Methods Funct. Theory, {\bf 6} (1) (2006), 15--28.

\bibitem{HIMPS07} {P. H\"ast\"o, Z. Ibragimov, D. Minda, S. Ponnusamy and S. K. Sahoo,}
{\em Isometries of some hyperbolic-type path metrics, and the hyperbolic medial axis},
In the tradition of Ahlfors-Bers, IV, Contemporary Math., {\bf 432} (2007), 63--74.

\bibitem{HIM08}{D. Herron, Z. Ibragimov and D. Minda,}
{\em Geodesics and curvature of M\"obius invariant metrics},
Rocky Mountain J. Math., {\bf 38} (3) (2008), 891--921.

\bibitem{Ibr03}{Z. Ibragimov,}
{\em On the Apollonian metric of domains in $\overline{\mathbb{R}^n}$}.
Complex Var. Theory Appl., {\bf 48} (10) (2003), 837--855.

\bibitem{Ibr2003}{Z. Ibragimov,}
{\em Conformality of the Apollonian metric},
Comput. Methods Funct. Theory, {\bf 3} (1-2) (2003), 397--411.

\bibitem{Ibr09}{Z. Ibragimov,}
{\em The Cassinian metric of a domain in $\bar{\mathbb{R}}^n$},
Uzbek. Mat. Zh., {\bf 1} (2009), 53--67.

\bibitem{KLVW} { R. Kl\'en, H. Linden, M. Vuorinen} and { G. Wang},
{\em The visual angle metric and M\"obius transformations},
Comput. Methods Funct. Theory, {\bf 14} (2-3) (2014), 577--608.
% {\bf Volume no.}(year), Issue no., 1--15.

\bibitem{KP94} {R. Kulkarni and U. Pinkall},
{\em A canonical metric for M\"obius structures and its applications,}
Math. Z., {\bf 216} (1994), 89--129.

\bibitem{Sei99}  {P. Seittenranta,}
{\em  M\"obius-invariant metrics},
Math. Proc. Cambridge Philos. Soc., {\bf 125} (1999), 511--533.

\bibitem{Vuo85} {M. Vuorinen,}
{\em Conformal invariants and quasiregular mappings},
J. Anal. Math., {\bf 45} (1985), 69--115.

\bibitem{Vuo88} {M. Vuorinen,}
{\em Conformal geometry and quasiregular mappings},
Lecture note in mathematics, Springer-Verlag,
Berlin, 1988.

\end{thebibliography}
\end{document}